\documentclass[12pt]{amsart}
\usepackage{tikz}
\usetikzlibrary{calc, arrows, decorations.markings} \tikzset{>=latex}
\usepackage{amsmath,amssymb,amsfonts}
\usepackage{amscd, amssymb, latexsym, amsmath, amscd}
\usepackage[all]{xy}
\usepackage{pb-diagram}
\usepackage{enumerate}
\usepackage{anysize}
\usepackage[sc]{mathpazo} 
\usepackage{enumitem}
\marginsize{2.5cm}{2.5cm}{2.5cm}{2.5cm}
\usepackage{graphics}
\usepackage{color}
\DeclareFontFamily{U}{wncy}{}
    \DeclareFontShape{U}{wncy}{m}{n}{<->wncyr10}{}
    \DeclareSymbolFont{mcy}{U}{wncy}{m}{n}
    \DeclareMathSymbol{\Sh}{\mathord}{mcy}{"58}

\marginsize{3cm}{3cm}{3cm}{3cm}
\input xy
\xyoption{all}
\usepackage{pb-diagram}
\usepackage[all]{xy}
\input xy
\xyoption{all}
\theoremstyle{plain}
\newtheorem{thm}{Theorem}

\newtheorem{lemma}{Lemma}

\newtheorem{corollary}{Corollary}

\newtheorem{proposition}{Proposition}

\newtheorem{conj}{Conjecture}

\theoremstyle{definition} \theoremstyle{definition}
\newtheorem{remark}{Remark}
\newtheorem{question}{Question}
\newtheorem{example}{Example}
\newtheorem{defn}[thm]{Definition}

\theoremstyle{remark}

\newcommand{\G}{\textsc{\G}}

\newcommand{\Q}{\mathbb{Q}}
\newcommand{\LG}{{}^L{\G}}

\newcommand{\ad}{{\rm ad}}

\newcommand{\U}{\mathcal{U}}
\newcommand{\Pa}{\mathcal{P}}
\newcommand{\Z}{\mathbb{Z}}

\newcommand{\R}{\mathbb{R}}

\newcommand{\C}{\mathbb{C}}
\newcommand{\Si}{\mathbb{S}}

\newcommand{\wG}{\widehat{G}}
\newcommand{\wM}{\widehat{M}}
\newcommand{\wB}{\widehat{B}}
\newcommand{\wP}{\widehat{P}}
\newcommand{\wN}{\widehat{N}}
\newcommand{\wZ}{\widehat{Z}}
\newcommand{\wT}{\widehat{T}}

\newcommand{\der}{\rm der}
\newcommand{\scon}{\rm sc}

\newcommand{\Ind}{{\rm Ind}}

\def\G{{\rm G}}
\def\Aut{{\rm Aut}}
\def\SL{{\rm SL}}
\def\Spin{{\rm Spin}}

\def\PSO{{\rm PSO}}
\def\PGSO{{\rm PGSO}}

\def\GSp{{\rm GSp}}
\def\PGSp{{\rm PGSp}}
\def\Sp{{\rm Sp}}

\def\U{{\rm U}}

\def\GL{{\rm GL}}

\def\Gal{{\rm Gal}}
\def\SO{{\rm SO}}
\def\OO{{\rm O}}
\def\Out{{\rm Out}}

\def\ad{{\rm ad\, }}

\begin{document}

\title[Generalizing the MVW involution, and the contragredient]
{Generalizing the MVW  involution, and the contragredient}
\author{Dipendra Prasad}
\address{Tata Institute of Fundamental 
Research, Homi Bhabha Road, Bombay - 400 005, INDIA.}
\email{dprasad@math.tifr.res.in}


\begin{abstract}
For certain quasi-split reductive groups $G$ over a general field $F$, we construct an automorphism $\iota_G$ of $G$
over $F$, well-defined as an element of $\Aut(G)(F)/jG(F)$ where $j:G(F) \rightarrow \Aut(G)(F)$ 
is the inner-conjugation action of $G(F)$ on $G$. The automorphism $\iota_G$ generalizes (although only for quasi-split groups) an involution due to Moeglin-Vigneras-Waldspurger in [MVW] for classical groups 
which takes any 
irreducible admissible  representation 
$\pi$ of $G(F)$ for $G$ a classical group and $F$ a local field, to its contragredient $\pi^\vee$.

The paper also formulates a conjecture on the contragredient
of an 
irreducible admissible  representation 
of $G(F)$ for $G$ a reductive algebraic group over   a local field $F$ 
 in terms of the (enhanced) Langlands parameter of the representation.

\end{abstract}

\maketitle
{\hfill \today}

\tableofcontents

\section{Introduction}
Let $G$ be a reductive algebraic group over   a local field $F$. The paper has two related goals. First, to construct 
an involutive automorphism $\iota_G$ of $G(F)$, generalizing an involution used by Moeglin-Vigneras-Waldspurger in [MVW] for 
classical groups which takes an irreducible admissible representation $\pi$ of $G(F)$ to  $\iota_G(\pi) \cong \pi^\vee$, where 
$\pi^\vee$ is the contragredient of $\pi$. We construct such an involution 
(well-defined as an element of $\Aut(G)(F)/jG(F)$ where $j:G(F) \rightarrow \Aut(G)(F)$ 
is the inner-conjugation action of $G(F)$ on $G$) for a certain class of quasi-split groups, in particular for
those for which $2H^1(\Gal(\bar{F}/F), Z(\Bar{F}))=0$ where $Z$ is the center of $G$.
This 
part of the paper --- although requisitioned by representation theory of $G(F)$ for $F$ a local field --- is purely an exercise
in algebraic groups valid for a general field $F$. 
An example to keep in mind of what we do is best explained by the group $\SO(2n)(F)$ for which $\Aut(G)(F)/jG(F)$ is a large group  
but has a particularly useful element for 
representation theory which is conjugation by any element of $\OO(2n)(F)$ which is not in $\SO(2n)(F)$. 

The involution $\iota_G$, even when it exists, does not necessarily take an irreducible admissible representation of $G(F)$ to its contragredient (see Conjecture 1 for a precise statement),
but perhaps it comes as close as it gets among automorphisms of $G(F)$ with this property, i.e. if there is an automorphism $a_\pi$ of $G(F)$ 
with the property that $a_\pi(\pi) \cong \pi^\vee$, then $\iota_G(\pi) \cong \pi^\vee$.

If $F=\R$, there is a 
description by J. Adams in [Ad] of the contragredient in terms of what he calls a Chevalley involution ---
which fixes a fundamental torus $H_f$ inside $G(\R)$ and acts on it by $t\rightarrow t^{-1}$. We check that our involution
coincides as an element of $\Aut(G)(\R)/G(\R)$ with what Adams constructs if $G$ is quasi-split over $\R$.
 
Secondly, we formulate 
a conjecture on 
the Langlands-Vogan parameter of the  contragredient $\pi^\vee$ 
of an irreducible admissible representation $\pi$ of $G(F)$ 
in terms of the Langlands-Vogan parameter of $\pi$. 
The description of the Langlands-Vogan parameter of $\pi^\vee$ has three ingredients. The first, the 
most obvious being that the parameter must go to the `dual parameter,' effected 
by the Chevalley involution on $\LG$; second, the character of the component group must
go to the dual character; finally, we have to take into account 
the effect on `base point':
since if  $\pi$  has a Whittaker model for a character $\psi: N \rightarrow \C^\times$, 
$\pi^\vee$  has a Whittaker model not for the character $\psi$, but for $\psi^{-1}$.

The two conjectural descriptions of the contragredient --- one in terms of an automorphism of $G(F)$, and the other in terms of the Langlands-Vogan parameter of $\pi$, the first intrinsic to $G(F)$, the second quite extrinsic 
to the structure of $G(F)$ ---  are related in some sense, although neither seems a consequence of the other. 
Actually, the description of the  contragredient in terms of the Langlands-Vogan parameter has larger applicability 
as it sees characters  of component groups which need not be selfdual in which cases the involution $\iota_G$ fails to 
have the desired effect.

\section{Notation and other preliminaries} 
In this paper, $G$ will always stand for a connected reductive algebraic group 
over a local field $F$, which can be either archimedean, or non-archimedean. 
By a connected real reductive 
group, we will mean $G(\R)$ for $G$ a connected reductive algebraic group defined over $\R$; thus $G(\R)$ might well be disconnected such as
$\SO(p,q)(\R)$. We will use $\Si$ to denote the group  of norm 1 elements in $\C^\times$, and $\Si^n$, its $n$-th power.  

Let $W_F$ be the Weil group of $F$, and $W'_F$ the Weil-Deligne group of $F$.
Let ${}^LG(\C) = \widehat{G}(\C) \rtimes W_F$
be the $L$-group of $G$ which comes equipped with a map onto
$W_F$. Often we may simplify notation by denoting ${}^LG(\C)$ by 
${}^LG$, and $\wG(\C)$ by $\wG$.

An admissible homomorphism $\varphi: W'_F\rightarrow {}^LG$
is called a Langlands parameter for $G(F)$.  To an admissible homomorphism 
$\varphi$ is associated the group of connected components $\pi_0(Z_{\varphi} )
= Z_{\varphi}/ Z^0_{\varphi}$ where ${Z_{\varphi}}$ is the centralizer of
$\varphi$ in $\widehat{G}(\C)$, 
and ${Z^0_{\varphi}}$ is its connected component containing the  
identity element. Two Langlands parameters $\varphi_1,\varphi_2: W'_F\rightarrow {}^LG(\C)$
are said to be equivalent if $\varphi_1(w)= g_0\varphi_2(w)g_0^{-1}$ for some $g_0\in \widehat{G}(\C)$.

An {\it enhanced Langlands parameter} is a pair $(\varphi,\mu)$ consisting of an admissible homomorphism 
$\varphi: W'_F\rightarrow {}^LG(\C)$, and an irreducible representation $\mu$ of 
its group of connected components $\pi_0(Z_{\varphi})$. 
According to the Langlands-Vogan parametrization, cf. [Vo], there should be  a bijective correspondence
(depending on fixing a base point, 
consisting of a pair $(\psi,N)$ where $\psi$ is a 
non-degenerate character of $N=N(F)$ which is a maximal unipotent subgroup in 
a fixed quasi-split  group $G_0(F)$) between enhanced Langlands parameter $(\varphi, \mu)$
and pairs $(\pi,G(F))$ consisting of a pure innerform $G$ of $G_0$ over $F$, and an irreducible
admissible representation $\pi$ of $G(F)$.
Beginning with the work of Harris-Taylor-Henniart on the local Langlands correspondence for $\GL_n(F)$, 
the Langlands-Vogan parametrization is now supposed to be known via the work of Arthur, Moeglin, and others 
for all classical groups.

In this paper, we will have occasions to use the Jacobson-Morozov theorem according to which, if $G$ is a reductive algebraic group
over a field $F$ of characteristic zero and $u \in G(F)$ is a unipotent element, then there is a homomorphism 
$\phi: \SL_2 \rightarrow G$ 
defined over $F$ taking the upper triangular unipotent matrix $ \left ( \begin{array}{cc} 1 & 1 \\ 0 &  1\end{array}\right ) \in \SL_2(F)$ 
to $u$. Given a unipotent element $u \in G(F)$,
the homomorphism $\phi$ is unique up to conjugation by $Z_G(u)$, the centralizer of $u$.

If $u$ is a regular unipotent element of $G(F)$ (which exists if and only if  $G$ is a quasi-split group over $F$), then the centralizer
of $u$ is up to the center of $G$, a unipotent group of dimension equal to the semi-simple rank of $G(\bar{F})$ contained inside the unique
Borel subgroup of $G$ in which $u$ lies. It follows that if 
$\phi: \SL_2 \rightarrow G$ corresponds to a regular unipotent element in $G(F)$, 
and takes the group of diagonal matrices
in $\SL_2$ inside a fixed maximal torus $T$ of $G$, then such a homomorphism is unique, and hence defined over $F$.
Therefore if $G$ is a quasi-split group over $F$, $(G,{B}, T, \{X_\alpha\})$   
is a pinning on $G$ defined over $F$, there is a 
unique homomorphism $\phi: \SL_2 \rightarrow G$ 
defined over $F$ corresponding to the regular nilpotent element $\sum X_\alpha$ in the Lie algebra of $G(F)$,
taking the group of diagonal matrices
in $\SL_2$ to $T$.

\begin{remark}It may be noted that for a quasi-split group $G$ over any field $F$, if $B=TN$ is a Borel subgroup of $G$ with 
unipotent radical $N$, and if we understand by an {\it algebraic Whittaker datum} a morphism of algebraic 
groups $\psi: N(F) \rightarrow F$ which restricted to 
any simple root space in $N(\bar{F})$ is nontrivial, then there is 
a bijective correspondence between algebraic Whittaker datum on $G$, and $N(F)$-conjugacy classes of 
pinnings $(G,{B}, T, \{X_\alpha\})$   on $G$ with $B$ the normalizer of $N$ in $G$, by sending an algebraic Whittaker datum
$\psi$ to the collection of elements $X_\alpha$  in the $\alpha$-root spaces in $N(\bar{F})$ 
with $\psi(X_\alpha) = 1$
for all simple roots $\alpha$. 
For defining
a homomorphism  $\psi: N(F)\rightarrow F$ to be non-degenerate,
there is no need to fix $T$. This is so because $\psi$ is defined on 
$N(\bar{F})/[N(\bar{F}),N(\bar{F})]$ which is a module for $B/N\cong T$ (`canonical' Cartan), hence decomposes 
as a sum of simple root spaces for $B/N \cong T$.

For $F$ a local field, 
fixing a nontrivial character $\psi_0:F\rightarrow \C^\times$ gives a bijective correspondence between 
 algebraic Whittaker datum and Whittaker datum on a reductive group $G$ over $F$.

\end{remark}  

\section{Constructing a duality involution}

The book [MVW] of Moeglin-Vigneras-Waldspurger constructs an involution $\iota_G$ for all classical groups $G(F)$ (defined without using
division algebras) which has the property that $\iota_G(\pi) \cong \pi^\vee$ for $\pi$ any irreducible admissible representation of the group $G(F)$. This
involution plays an important role in many results on classical groups, and has often been called the MVW involution. 
The most well-known of these involutions is for $\GL_n(F)$ in which case it can be taken to be $g\rightarrow {}^tg^{-1}$. 
It is a theorem of Gelfand-Kazhdan that this involution takes an irreducible admissible  representation of $\GL_n(F)$ 
to its contragredient. For $\Sp_{2n}(F)$, the involution is the conjugation by an element of $\GSp_{2n}(F)$ with similitude $-1$. Eventually, the
MVW theorem is proved by a group theoretic assertion on classical groups that $\iota_G(g)$ is conjugate to $g^{-1}$ for all  $
g \in G(F)$.

The aim of this section is to construct an involutive automorphism $\iota_G$ of  $G(F)$ for any quasi-split group $G$ 
over {\it any} field $F$ which coincides with the MVW involution for classical groups which we hope has the property that it takes
an irreducible admissible representation $\pi$ of $G(F)$, $F$ a local field,  to its contragredient (for certain groups $G$ and certain representations $\pi$ of $G(F)$ to be made precise later).

At the very least, for certain quasi-split reductive groups $G$, 
we will construct an involution $\iota_G$ on $G(F)$ which is well-defined as an element of
$\Aut(G)(F)/G(F)$ 
where $\Aut(G)(F)$ is the group of automorphisms of $G$ defined over $F$, and by abuse of notation,
we denote here and elsewhere in the paper $\Aut(G)(F)/G(F)$ to be the quotient group of $\Aut(G)(F)$ by the group
of inner automorphisms of $G$ coming from elements of $G(F)$. (Abuse of notation because the latter group is not $G(F)$ 
but is $G(F)/Z(F)$.) 
Since the involution is being constructed to take a representation to its contragredient, its being well-defined as an element in $\Aut(G)(F)/G(F)
$ is all that matters.

The involution $\iota_G$ will be defined using certain auxiliary data and later we will check that it is well-defined as 
an element of $\Aut(G)(F)/G(F)$ 
in certain cases; it is only for such groups $G$ for which $\iota_G$ is well-defined as an
element of $\Aut(G)(F)/G(F)$ that $\iota_G$ plays a particularly useful role, and it is only in such cases that 
we will use it in this paper.

First recall that for any reductive algebraic group $G$ over an algebraically closed field $F=\bar{F}$, with $T$ as a maximal torus,
a Chevalley involution is an involutive automorphism of $G$ preserving $T$ and acting as $t\rightarrow t^{-1}$ on $T$. 
Existence of a Chevalley 
involution is a well-known theorem. A Chevalley involution takes any irreducible
algebraic representation of $G$ to its contragredient.  Having fixed $T$, a Chevalley involution 
considered as an element in $\Aut(G)$ is unique up to multiplication  
by an inner automorphism of $G$ by an element of $T$.
We now fix a Borel subgroup $B$ containing $T$, and will
find it more convenient to multiply a Chevalley involution by an automorphism of $G$
 induced by conjugation by $\omega_G$,
where $\omega_G$ is a lift in $G$ (in fact in $N(T)$, the normalizer of $T$ in $G$) 
of the longest element in the Weyl group of $T$ which takes $B$ to the opposite Borel $B^-$. The resulting automorphism of $G$ 
will act on $T$ as $t\rightarrow w_G(t^{-1})$,  will preserve $B$, and by further 
multiplication by an automorphism of $G$ induced by conjugation by an element of $T$, 
we can assume that the automorphism 
fixes a given pinning $\Pa= (G,{B}, T, \{X_\alpha\})$ of $G$, 
constructing  a uniquely defined  involutive automorphism $c_{G,\Pa}$ 
(perhaps trivial) of $G$, to be called the Chevalley involution associated to the pinning $\Pa= (G,{B}, T, \{X_\alpha\})$ of $G$. 
Since any two pinnings on $G$ are conjugate under $G(\bar F)$, this involutive automorphism is uniquely defined up to
conjugation by an element of $G(\bar F)$.

We will now define Chevalley involution for any quasi-split reductive group over a general field $F$ which comes equipped 
with a pinning  $(G,{B}, T, \{X_\alpha\})$.

We begin with a split group $G_0$ now over a general field $F$ with a pinning  $(G_0,{B_0}, T_0, \{X_\alpha\})$ over $F$.
It is known that $$\Aut(G_0,{B_0}, T_0, \{X_\alpha\})(F) = \Aut(G_0(\bar{F}),{B_0}(\bar F), T_0(\bar F), \{X_\alpha\}) 
= \Out (G_0(\bar{F})),$$
where $\Out (G_0(\bar{F})),$ is defined by the exact sequence,
$$1 \rightarrow {\rm Int}({ G_0})(\bar{F})  \rightarrow {\Aut}({ G_0})(\bar{F}) 
\rightarrow {\Out}( G_0(\bar{F})) 
\rightarrow 1.$$

A quasi-split group over $F$ is described by a (conjugacy class of) homomorphism 
$$\lambda: \Gal(\bar{F}/F) \rightarrow \Aut(G_0({F}),{B_0}( F), T_0( F), \{X_\alpha\}) = \Out (G_0(\bar{F})),$$
allowing one to twist $G_0$ by the cocycle $\lambda \in H^1(\Gal(\bar{F}/F),\Aut(G_0))$
to construct a group $G_\lambda$ over $F$. The twisted group $G_\lambda$ continues to have $B_0(\bar{F})$ and $T_0(\bar{F})$ Galois invariant, and hence define subgroups $B_\lambda$ and $T_\lambda$ of $G_\lambda$ over $F$; further, the set 
of roots $\{X_\alpha\}$ is Galois invariant 
(although not pointwise).

The automorphism group of a quasi-split reductive group $G_\lambda$ over $F$  fixing the pinning
$(G_\lambda,{B_\lambda}, T_\lambda, \{X_\alpha\})$ 
are the elements of $\Aut(G_0(\bar{F}),{B_0}(\bar F), T_0(\bar F), \{X_\alpha\}) 
= \Out (G_0(\bar{F}))$ 
which commute with the image of $\lambda$. The Chevalley involution $c_{G_0,\Pa}$ belongs to the center of 
$\Aut(G_0(\bar{F}),{B_0}(\bar F), T_0(\bar F), \{X_\alpha\})$, hence induces an involution of $G_\lambda$ defined over $F$.
 
This allows one to introduce an automorphism $c_{G,\Pa}$ 
of $G = G_\lambda$ defined over $F$, 
 fixing a pinning $\Pa= (G,{B}, T, \{X_\alpha\})$ 
and which acts as $t\rightarrow w_G(t^{-1})$ on $T$ where $w_G$ is the longest element in the Weyl group of $T$ taking ${B}$ to ${B}^-$, 
the opposite Borel subgroup. We call $c_{G,\Pa}$, the Chevalley involution associated to a 
pinning $\Pa$ on $G$. Observe that $c_{G, \Pa}=1$ if and only if $-1 \in W_G$, thus in all simple groups excepts
those of type $A_n, D_{2n+1}, E_6$, $c_{G, \Pa}=1$.

\begin{defn} (Duality involution) For a quasi-split reductive algebraic group $G$ over a general field $F$, fix a pinning 
${\mathcal P}= ( G, B, T, \{X_\alpha\}) $ on $G$, and let $c_{G,\Pa}$ be the Chevalley involution constructed above. 
Define an automorphism $\iota_{G,\Pa} \in \Aut(G)(F)$ 
-- to be called the duality involution -- as a product of two commuting automorphisms $\iota_{G,\Pa}=\iota_{-}c_{G,\Pa}$ 
where $\iota_{-} \in T^{\ad}(F)$
is the unique element which  acts by $-1$ on all simple root spaces of $T$ in $B$ which is considered as an element of $\Aut(G)(F)$ by the inner conjugation action. 
\end{defn} 
  
\begin{example} We leave the pleasant  task of identifying $\iota_{G,\Pa}=\iota_{-}c_{G,\Pa}$ with the MVW involution for any of the classical quasi-split groups $G= \Sp_{2n}, \SO_{2n}, \U_n$ to the reader except to point out that for $\Sp_{2n}$, $c_{G,\Pa}=1$;
for $\SO_{2n}$, $\iota_{-} \in T(F)$, and $c_{G,\Pa}=1$ if $n$ is even; for $\U_n$, $\iota_{-} \in T(F)$ if $n$ is odd. 
\end{example}

Since any two pinnings $\Pa,\Pa'$ on $G$ are conjugate under $G^{\ad}(F) = (G/Z)(F)$, the corresponding duality involutions,  $\iota_{G,\Pa}$ 
and $\iota_{G,\Pa'}$ are related as: $\iota_{G,\Pa'} = g \iota_{G,\Pa} g^{-1}$ where $g \in G^{\ad}(F)$ 
is to be treated as an element of $\Aut (G)(F)$.
In more detail, $$\iota_{G,\Pa'}(x) = g \iota_{G,\Pa}(g^{-1}xg)g^{-1} = g \iota_{G,\Pa}(g^{-1}) \iota_{G,\Pa}(x) 
\iota_{G,\Pa}(g)g^{-1}, \quad \forall x \in G(\bar{F}). \,\,\, \hspace{.5cm}{(*)}$$ 

Since $G(F)$ operates transitively on pairs $(B,T)$, and for a given $(B,T)$, $T^{\ad}(F)$ operates transitively 
on the choice of pinnings $\{X_\alpha \}$,
it follows that $G^{\ad}(F)=G(F)T^{\ad}(F)$; we record this in the following well-known lemma.

\begin{lemma} \label{springer}
For a quasi-split reductive group $G$ over a field $F$ with center $Z$, if $B$ is a Borel subgroup of $G$, and $T$ a maximal torus of $B$, then for $G^{\ad}=G/Z$, and $T^{\ad}=T/Z$,
we have an isomorphism of groups:
$$T^{\ad}(F)/T(F) \cong G^{\ad}(F)/G(F),$$
equivalently, $G^{\ad}(F) = T^{\ad}(F)G(F)$.
\end{lemma}

Write $g \in G^{\ad}(F)$ as $g = g_0t$ with $g_0 \in G(F)$, and $T\in T^{\ad}(F)$. Then,
$ g \iota_{G,\Pa}(g^{-1}) = g_0 t \iota_{G,\Pa}(t^{-1}) 
\iota_{G,\Pa}(g_0^{-1})$. Since the action of $\iota_{G,\Pa}$ on $T$ is given by $t \rightarrow 
w_G(t^{-1})$, 
$t \iota_{G,\Pa}(t^{-1}) = t w_G(t)$. Therefore if we can prove that $tw_G(t) \in T(F)$ for all $t \in T^{\ad}(F)$, 
it would follow from $(*)$ that for different choices of $\Pa$, the automorphisms $\iota_{G,\Pa}$ differ by $G(F)$.
In fact, the  automorphisms $\iota_{G,\Pa}$, as we vary the pinning $\Pa$, differ by $G(F)$ if and only if $
tw_G(t) \in T(F)$ for all $t \in T^{\ad}(F)$.

From 
the short exact sequence of commutative algebraic groups,  
$1 \rightarrow Z \rightarrow T \rightarrow T^{\ad} \rightarrow 1$,
we get the long exact sequence of Galois cohomology groups:
$$1 \rightarrow { T^{\ad}(F)/T(F)}  \rightarrow H^1(F,Z) 
\rightarrow H^1(F,T). $$
It follows that $tw_G(t) \in T(F)$ for all $t \in T^{\ad}(F)$ if and only if  
the automorphism of $H^1(F,Z)$ arising from the automorphism $z\rightarrow w_G(z^{-1})$ of $Z$ is 
trivial on ${\rm ker} \{H^1(F,Z) \rightarrow H^1(F,T)\}$. Since 
the action of $w_G$ on $Z$ is trivial, the automorphism $z\rightarrow w_G(z^{-1})$ of $Z$ is nothing but
$z \rightarrow z^{-1}$. Hence the action of the automorphism $z\rightarrow w_G(z^{-1})$ on $H^1(F,Z)$ is nothing
but $a\rightarrow -a$ (using additive notation now).

We record this conclusion as a proposition.

\begin{proposition}\label{prop2}
For $G$ a quasi-split group over a field $F$, the automorphism $\iota_{G,\Pa}$ constructed using a pinning $\Pa$ on $G$ is independent of the pinning
chosen (as an element of $\Aut(G)(F)/G(F)$) if ond only if 
$2{\rm ker}\left \{H^1(F,Z) \rightarrow H^1(F,T) \right \} =0$, in particular if 
 $2H^1(F,Z(G))=0$.
\end{proposition}

\begin{corollary} If $F=\R$, or $Z=Z(G)$ is an elementary abelian 2-group, or 
$Z=Z(G)$ is a split torus, or more generally an induced torus,  
then the duality involution 
$\iota_{G,\Pa}$ (as an element of $\Aut(G)(F)/G(F)$) 
is independent 
of the pinning $\Pa$.
\end{corollary}

\begin{example} The condition $2H^1(F,Z(G))=0$ in Proposition \ref{prop2} is satisfied for all classical 
groups $G=\GL_n, \U_n, \SO_n, \Sp_{2n}$. On the other hand, it is not satisfied for $\SL_n$, and 
 as we discuss in some detail in Example \ref{example1}, there is no good duality involution on representations
of $\SL_n(F)$.
\end{example}

Using Remark 1 relating pinnings to algebraic Whittaker datum, the involution $\iota_{G,\Pa}$ can also be constructed
using algebraic Whittaker datum. We quickly recall this.

Let $\psi: N(F) \rightarrow F$ be 
an algebraic Whittaker datum.
The group $G$ has an automorphism $\iota_{B,T,N,\psi} $ 
defined over $F$ 
which takes the pair $(T,B)$ to itself, $\psi$ to $\psi^{-1}$, and  operates on $T$ as $t\rightarrow w_G(t^{-1})$;
its effect on simple roots of $T$ on 
$N(\bar{F})$ is that of the automorphism $-w_G$ where $w_G$ is the longest element in the Weyl group
of $G$ over $\bar{F}$, i.e., $\iota_{B,T,N,\psi}(\alpha) = -w_G(\alpha)$ for all simple roots of $T$ inside $N(\bar{F})$.
The automorphism $\iota_{B,T,N,\psi} $ can be taken to be  $\iota_{G,\Pa}$ for the pinning $\Pa$ associated to the 
algebraic Whittaker datum
by Remark 1.

\begin{conj} \label{conj}
For $G$ a quasi-split group over a local field $F$, let $\iota_G$ be the
duality automorphism of $G$ constructed above which  we assume is independent of the pinning used to define it 
as an element of $\Aut(G)(F)/G(F)$ (thus imposing conditions on $G$ such as in proposition \ref{prop2}). Let $\pi$ be an irreducible admissible generic representation of $G(F)$ (for any Whittaker datum),
then $\pi^\vee \cong \iota_G(\pi)$. 
For $\pi$ any irreducible admissible representation of $G(F)$, if the Langlands parameter $\varphi$ associated to 
$\pi$ has the property that  $\pi_0(Z_{\varphi})$ 
is an elementary abelian 2-group (or its irreducible representations are selfdual), $\pi^\vee \cong \iota_G(\pi)$. 
\end{conj}

We end this section with the following question.

\begin{question} If $G_0,G$ are reductive groups over a non-archimedean 
local field $F$ which are innerforms of each other,
it follows for example from Lemma 3.1 
of Kaletha's paper [Ka] that
$G_0^{\ad}(F)/G_0(F)$ is {\it naturally} isomorphic to $G^{\ad}(F)/G(F)$.
It is possible that if $G_0$ is a quasi-split group, then 
there is a natural injection of the `outer automorphism groups' $\Aut(G)(F)/G(F) \hookrightarrow \Aut(G_0)(F)/G_0(F)$,
with explicit knowledge of when there is an equality (perhaps when $G$ is obtained as a pure innerform of $G_0$, i.e., 
when $G$ is obtained by twisting with respect to a cohomology class in $H^1(F,G_0)$?). If the group $G$ has no outer automorphism over $\bar{F}$, 
then $\Aut(G)(F)=G^{\ad}(F)$, and  Lemma 3.1 of [Ka] mentioned in the beginning of this paragraph  
already answers the question.
If $\Aut(G)(F)/G(F) \cong \Aut(G_0)(F)/G_0(F)$, then the duality automorphism
constructed in this paper as an element of the group   $\Aut(G_0)(F)/G_0(F)$ can be transported to an element of 
$\Aut(G)(F)/G(F)$, and may play a role similar to what it does here, 
namely for any irreducible admissible representation $\pi$ of $G(F)$, if the Langlands parameter $\varphi$ associated to 
$\pi$ has the property that  $\pi_0(Z_{\varphi})$ 
is an elementary abelian 2-group (or its irreducible representations are selfdual), $\pi^\vee \cong \iota_G(\pi)$; [MVW] as well as the work of Adams [Ad] 
in the real case supports this expectation.   (Even for quasi-split groups we are not suggesting that 
$\iota_G(\pi) \cong \pi^\vee$, so neither will we do so for general groups.) 
\end{question}

\section{Conjecture on the contragredient}

We begin with the following basic lemma regarding the contragredient.

\begin{lemma} \label{Whit} Let $G(F)$ be a quasi-split reductive algebraic group over a local field $F$ of characteristic 0, with $B=TN$ a Borel subgroup of $G$,
and $\psi: N(F)\rightarrow \C^\times$ a non-degenerate character. Then if an irreducible admissible representation $\pi$ of $G(F)$
has a Whittaker model for the character $\psi$, then the contragredient $\pi^\vee$ has a Whittaker model for the character
$\bar{\psi}= \psi^{-1}: N \rightarrow \C^\times$.
\end{lemma}

\begin{proof}Although this lemma must be a standard one, the author has not found a proof in the literature, so here is one.
Observe that if $\pi$ is unitary, then $\pi^\vee$ is nothing but the complex conjugate $\bar{\pi}$ of $\pi$, i.e., 
$\bar{\pi} = \pi \otimes_{\C} \C$ where $\C$ is considered as a $\C$ module using the complex conjugation action. By applying complex conjugation to a  linear form $\ell: 
\pi \rightarrow \C$ on which $N(F)$ operates by $\psi$, we see that $\bar{\ell}$ defines a $\bar{\psi} = \psi^{-1}$-linear form on $\bar{\pi} = \pi^\vee$, proving the lemma for unitary representations, in particular for representations of any reductive group which are tempered up to a twist. 
The general case of the lemma follows by realizing a generic representation of $G(F)$ as an irreducible  principal 
series representation 
$\pi = {\rm Ind}_{P(F)}^{G(F)}\mu$ from a  representation $\mu$ which is tempered up to a twist of a Levi subgroup $M(F)$ of $P(F)$; that this
can be done is what's called the {\it standard module conjecture}, for a proof of which see
 corollary 1.2 in [HO] for $p$-adic fields, and [Vo2] for $F=\R$. (It is in this appeal to 
 [HO] that we use the characteristic zero hypothesis in the statement of the lemma.)

Now note that if $N_M$ is the unipotent radical of a Borel subgroup in $M$,
then there is a general recipe due to Rodier (recalled in the 4th paragraph of section \ref{Rodier}), 
 constructing a character $\psi_M:N_M(F)\rightarrow \C^\times$ from a  character $\psi: N(F)\rightarrow \C^\times$ with the property that 
  the induced principal series ${\rm Ind}_{P(F)}^{G(F)}\mu$
is $\psi$-generic if and only if  $\mu$ is $\psi_M$-generic.  The recipe of Rodier has the property that if $\psi_M:N_M(F)\rightarrow \C^\times$ is 
associated to the  character $\psi: N(F)\rightarrow \C^\times$ then
$\bar{\psi}_M:N_M(F)\rightarrow \C^\times$ is associated to  $\bar{\psi}: N(F)\rightarrow \C^\times$.
This proves the lemma.
\end{proof}

We next recall that in [GGP], section 9, denoting $G^{\rm ad}$ the adjoint group of $G$ 
(assumed without loss of generality at this point to be quasi-split since we want to construct 
something for the $L$-group),  there is constructed a homomorphism 
$  {G}^{\rm ad}(F)/G(F) \rightarrow 
 {\pi_0(Z_{\varphi})}^\vee$, denoted $g \mapsto \eta_g$ which has the following (conjectural) 
property on character of component groups associated to 
a representation $\pi$: $\mu(\pi^g) = \mu(\pi) \otimes \eta_g$.
We recall the map $g\rightarrow \eta_g$ here. For doing this, let 
$\widehat{G}^{\scon}
$ be the universal cover of $\widehat{G}$:
$$1 \rightarrow 
 \pi_1(\widehat{G}) 
\rightarrow \widehat{G}^{\scon} 
\rightarrow \widehat{G} \rightarrow 1 . $$
By the definition of universal cover, any automorphism of $\widehat{G}$ lifts to an automorphism of $\widehat{G}^{\scon}$ uniquely, 
and thus if we are given an action of $W_F$ on $\widehat{G}$ (through a parameter $\varphi: W_F\rightarrow {}^L{G}$),
 it lifts uniquely to an action of $W_F$ on $\widehat{G}^{\scon}$,
preserving  $ \pi_1(\widehat{G})$. Treating the above as an exact sequence of $W_F$-modules, and taking 
$W_F$-cohomology, we get the boundary map:
$$ \widehat{G}^{W_F} = Z_\varphi \rightarrow H^1(W_F,  \pi_1(\widehat{G})),$$ 
which gives rise to the map $$ \pi_0(Z_\varphi) \rightarrow H^1(W_F,  \pi_1(\widehat{G})),$$
and corresponds to the map $Z_\varphi \times W_F \rightarrow \pi_1(\wG)$ which is $(z,w)\rightarrow w(\tilde{z})\cdot \tilde{z}^{-1}$
where $\tilde{z}$ is an arbitrary lifts of an element  $ z \in Z_\varphi$ to $\wG^{\scon}$.

On the other hand by the Tate duality (using the identification $Z(G)^\vee = \pi_1(\wG)$), 
there is a perfect pairing:
$$H^1(W_F, Z(G)) 
\times H^1( W_F,  \pi_1(\widehat{G})) \rightarrow \Q/\Z.$$
This pairing 
allows one to  think of elements in $H^1(W_F, Z(G)) $ as characters on $H^1(W_F,\pi_1(\widehat{G}))$, and therefore 
using the map $ \pi_0(Z_\varphi) \rightarrow H^1(W_F,  \pi_1(\widehat{G})),$ as characters on $\pi_0(Z_\varphi)$.
Finally, given the natural map $G^{\ad}(F)/G(F) \rightarrow H^1(F,Z(G))$, we have constructed a group homomorphism from
$G^{\ad}(F)/G(F)$ 
to characters on  $\pi_0(Z_\varphi)$.

\begin{remark}
For later use, we compare the homomorphism $G^{\ad}(F)/G(F) \rightarrow \pi_0(Z_\varphi)^\vee$ with a similar homomorphism constructed for a Levi subgroup in case the parameter $\varphi: W'_F\rightarrow {}^LG$
factors through ${}^LM$ as $ W'_F\stackrel{\varphi'}\rightarrow {}^LM \hookrightarrow {}^LG$. (We will continue to assume  $G$ to be a quasi-split group.)

For this,  we begin by noting that although there is no obvious map between  $G^{\ad}(F)/G(F)$ and $M^{\ad}(F)/M(F)$, if $Z$ is the centre of $G$, and $T$ a maximal torus in $G$ containing the maximally
split torus, the natural map from $(T/Z)(F)/T(F)$ to $G^{\ad}(F)/G(F)$  
is an isomorphism (cf. Lemma \ref{springer}).
Since there is a natural map from  $(T/Z)(F)/T(F)$ to $M^{\ad}(F)/M(F)$, we have a natural map from   $G^{\ad}(F)/G(F)$ to $M^{\ad}(F)/M(F)$. This gives rise 
to commutative diagrams of natural homomorphisms,
$$\xymatrixrowsep{1in}
\xymatrixcolsep{1in}
\xymatrix{  {G}^{\ad}(F)/G(F) \ar[r] \ar@{->}[d] 
& \ar[d]H^1(F,Z(G))  \ar[r] \ar@{->}[d]  & \ar[d] \pi_0(Z_\varphi)^\vee\\
 M^{\ad}(F)/M(F)\ar[r]  & H^1(F,Z(M))\ar[r] & \pi_0(Z_{\varphi'})^\vee .
}$$
\end{remark} 
$\hfill \Box$

  Let $g_0$ be the unique conjugacy class in $G^{\rm ad}(F)$ 
representing   an element in $T^{\rm ad}(F)$ (with $T^{\rm ad}$ a maximally split, maximal torus in $G^{\rm ad}(F)$) which acts by $-1$ on all simple root spaces of $T$ on $B$. Denote the corresponding $\eta_{g_0}$ by
$\eta_{-1}$, a character on $\pi_0(Z_{\varphi})$, which will be the trivial character 
for example if $g_0$ can be lifted to $G(F)$.

Let $c_{\wG}$ be the automorphism of $\wG$ preserving  a fixed pinning 
$(\wG,\widehat{B}, \wT, \{X_\alpha\})$   of $\wG$ 
and acting as $t\rightarrow w_{\wG}(t^{-1})$ on $\wT$ where $w_{\wG}$ is the longest element in the Weyl group of $\wT$ taking $\widehat{B}$ to $\widehat{B}^-$, 
the opposite Borel subgroup. Observe that $c_{\wG}=1$ if $-1 \in W_G=W_{\wG}$, thus in all simple groups excepts
those of type $A_n, D_{2n+1}, E_6$, $c_{\wG}=1$. 

The automorphism $c_{\wG}$ 
defined using a fixed pinning 
$(\wG,\widehat{B}, \wT, \{X_\alpha\})$   of $\wG$ 
belongs to the center of the group of diagram automorphisms of $\widehat{G}$, 
and hence extends to an 
automorphism of  ${}^LG(\C) = \widehat{G}(\C) \rtimes W_F$, which will again be denoted by 
$c_{\wG}$, and will be called the Chevalley involution of $\LG$.

\begin{conj} \label{conj2} Let $G$ be a reductive algebraic group over a local field $F$ which is a pure innerform of a 
quasi-split group $G_0$ over $F$ which comes equipped with a fixed triple $(B_0,N_0,\psi_0)$ used to parametrize 
representations on all pure innerforms of $G_0$. 
For an irreducible admissible representation $\pi$ of $G(F)$
with Langlands-Vogan parameter $(\varphi,\mu)$, the Langlands-Vogan parameter 
of $\pi^\vee$ is $(c_{\wG}\circ \varphi, (c_{\wG}\circ\mu)^\vee \otimes \eta_{-1})$, where $c_{\wG}$ is the Chevalley involution 
of ${}^L{G}(\C)$. (Since the Chevalley involution $c_{\wG}$ acts as $z\rightarrow z^{-1}$ on $Z(\wG)$, it follows 
that the character of the component group $c_{\wG}\circ\mu)^\vee \otimes \eta_{-1}$ 
defines the same pure innerform of $G_0$ as that defined by $\mu$.)

\end{conj}

\begin{remark} A consequence of the conjecture is that 
if the component groups are known to be elementary abelian 2-groups 
for a quasi-split 
reductive group $G_0(F)$, such as for any reductive group if $F=\R$,
or for classical groups defined using fields alone for any local field $F$, 
then all irreducible admissible representations 
of $G(F)$, a  pure innerform of $G_0(F)$, are selfdual if and only if $-1$ belongs to the Weyl group of $\wG$ 
(equivalently $c_{\wG} = 1$), and there is an 
$\iota_{-} \in T_0(F)$, where $T_0$ is a maximal torus in $B_0$, which acts by $-1$ on simple roots of $T_0$ in $B_0$.
\end{remark}

\begin{example} \label{example1}
We explicate  conjecture \ref{conj2} for $\SL_n(F)$ which in this case is an easy exercise. 
Let $\ell_a$ denote the natural (outer automorphism) action 
of $F^\times/F^{\times n}$ on $\SL_n(F)$ (corresponding to the conjugation action of 
$\GL_n(F)/F^\times \SL_n(F) \cong F^\times/F^{\times n}$ on $\SL_n(F)$, up to 
inner conjugation action of $F^\times \SL_n(F)$ on $\SL_n(F)$). Clearly, $(\ell_a \pi)^\vee \cong \ell_a \pi^\vee$ for any irreducible representation $\pi$ of $\SL_n(F)$. 
Now fix an irreducible representation $\pi$ 
of $\SL_n(F)$ which has a Whittaker model for the character $\psi:N \rightarrow \C^\times$ where $N$ is the group of upper triangular unipotent matrices in $\SL_n(F)$ and $\psi(n) = \psi_0(n_{1,2}+ \cdots + n_{n-1,n})$ for a fixed 
nontrivial character $\psi_0$ of $F$. Then by Lemma \ref{Whit}, $\pi^\vee$ 
has a Whittaker model for $\psi^{-1}$. Thus, if $\tau$ is the outer automorphism of $\SL_n(F)$ given by 
$\tau(g) = J {}^t\!g^{-1} J^{-1}$
where $J$ is the anti-diagonal matrix with each entry 1, 
 so that $\tau$  takes $(B,N,\psi)$ to $(B,N,\psi^{-1})$, then $\tau(\pi) = \pi^\vee$, 
by a combination of the Gelfand-Kazhdan theorem regarding the contragredient of an irreducible admissible representation  of 
$\GL_n(F)$ and the uniqueness of Whittaker model in an $L$-packet 
(i.e., given $\psi: N \rightarrow \C^\times$, there is a unique member in an $L$-packet of $\SL_n(F)$ with this Whittaker model), since $\pi^\vee$ is the unique irreducible representation of $\SL_n(F)$ in its $L$-packet with Whittaker model by $\psi^{-1}$. Having understood the action of $\tau$ on $\pi$, how about on other members $\ell_a\pi$ of the $L$-packet of $\pi$? Observe that, $\tau \circ \ell_a = \ell_{a^{-1}} \circ \tau$ (up to an element of $\SL_n(F)$). 
Therefore, $\tau(\ell_a \pi)= \ell_{a^{-1}}(\tau \pi) = \ell_{a^{-1}}(\pi^\vee)$. The upshot is that although
$\tau$ takes $\pi$ to its contragredient $\pi^\vee$, it does not take $\ell_a (\pi)$ to its contragredient, but 
to the contragredient of $\ell_{a^{-1}}( \pi)$. Our conjectures say exactly this for $\SL_n(F)$, and in general too it says something similar except that we cannot prove it for other groups since unlike for $\SL_n(F)$ where Whittaker models can be used to describe all members in an $L$-packet of 
representations,  this is not the case in general. (The change that we see for $\SL_n(F)$ from $\ell_a$ to $\ell_{a^{-1}}$ is in general the dual representation on the component group).
\end{example}

\begin{remark}
Conjecture \ref{conj2} allows for the possibility,
for groups such as  $\G_2, F_4$, or $E_8$ to have 
non-selfdual representations arising out of component groups (which can be 
$\Z/3,\Z/4$, and $\Z/5$ in these respective cases). 
Indeed $\G_2, F_4$, and  $E_8$ are known to have non 
selfdual representations over $p$-adic fields 
arising from compact induction of certain cuspidal unipotent 
representations of corresponding
finite groups of Lie type. {\it An important special case of both the conjectures 1 and 2
asserts that generic representations of these groups are selfdual; if the representation happens to be generic and supercuspidal, we are totally at a loss on this question.} (Global methods seem helpful here: globalizing
supercuspidal generic representations $\pi$ and $\pi^\vee$ to 
globally generic cuspidal representations unramified at all finite primes except this one, 
strong multiplicity one which one expects for globally generic representations, would prove $\pi \cong \pi^\vee$.)
\end{remark}

\begin{remark}
For a non-archimedean local field $F$, 
the projective (or, adjoint) symplectic group $\PGSp_{2n}(F)$ 
has a unique nontrivial pure innerform defined using the unique quaternion division algebra $D$ over 
$F$, call it $\PGSp_n(D)$.
Every irreducible  representation of $\PGSp_{2n}(F)$ is selfdual (by  a mild variation 
of the  theorem on page 91 in [MVW] for $\Sp_{2n}(F)$ to $\GSp_{2n}(F)$ according to which for any irreducible
admissible representation $\pi$ of $\GSp_{2n}(F)$, $\pi^\vee \cong \pi \otimes \omega_\pi^{-1}$ where $\omega_\pi$ is the central character of $\pi$
treated as a character of $\GSp_{2n}(F)$ using the similitude character $\GSp_{2n}(F)\rightarrow F^\times$.)
This is in conformity with our conjecture as $\PGSp_{2n}(F)$ is an adjoint group, 
with no diagram automorphisms; the $L$-group in this case is $\Spin_{2n+1}(\C)$, and the possible component groups are easily seen to be extensions of  subgroups of the component groups for $\SO_{2n+1}(\C)$ (which are elementary abelian 2-groups) by $\Z/2$, but since we are looking at $\PGSp_{2n}(F)$, the character of the component group is supposed to be trivial on the $\Z/2$ 
coming from the center of $\Spin_{2n+1}(\C)$, therefore such representations of the component group are actually representations of an elementary abelian 2-group, in particular selfdual.

In the work \cite{LST}, the authors prove that there is no analogue of MVW theorem for $\Sp_n(D), n \geq 3$, or for $ \SO_n(D), n \geq 5$. Looking at their proof, 
it is clear that their argument also proves that not every irreducible  representation of $\PGSp_{n}(D)$, or of $\PGSO_n(D)$ 
is selfdual. How does this compare with our conjecture \ref{conj2}? The only way out is to have more complicated component groups for 
$\Spin_{2n+1}(\C)$, in particular having non-selfdual representations for the component groups. 
The author has not seen any literature
on component groups for Spin groups. Since $\Spin_5(\C) \cong \Sp_4(\C)$, the component groups are abelian in this case, 
therefore the component group in $\Spin_{2n+1}(\C)$ 
could have non-selfdual representations only for $n \geq 3$, a condition consistent with [LST]. 
\end{remark}

\section{Reduction to the tempered case}\label{Rodier}

The aim of this section is to prove by using standard techniques that it suffices to prove Conjecture 2 about the contragredient for tempered representations. For doing this, we begin
with a general discussion of (enhanced) Langlands parameters of Langlands quotients.

Note that if an irreducible admissible representation $\pi$ of a reductive group $G$ is obtained as 
the Langlands quotient of a principal series $\Ind_P^G\sigma$, then the
Langlands parameter of $\pi$ is the same as the Langlands parameter of $\sigma$ considered as a parameter in ${}^LG$ via the natural embedding of the $L$-group
${}^LM$ in ${}^LG$ of a standard Levi subgroup $M$ of $G$. Here we remind ourselves  that since
$\pi$ is  the Langlands quotient of the principal series representation $\Ind_P^G\sigma$,  
the exponent of the representation $\sigma$ of $M$ (i.e., the absolute value of the character by 
which the maximal split central torus of $M$ acts on $\sigma$)   
is strictly positive. 
Under this condition, the centralizer of such a parameter $\varphi$ inside $\wM$ is the same as 
its centralizer in $\wG$. Calling these two centralizers as $Z_\varphi(\widehat{M})$ and $Z_\varphi(\wG)$, 
we have $Z_\varphi(\widehat{M}) = Z_\varphi(\wG)$, and hence an equality of the group of their connected components 
$\pi_0(Z_\varphi(\widehat{M})) = \pi_0(Z_\varphi(\wG))$, as well as 
of the set of irreducible representations of their component groups. 
 
 To have the enhanced local  Langlands correspondence, 
 we also need to fix Whittaker datum for the quasi-split groups $G$ and $M$. For this, we fix a 
 Borel subgroup $B=TN$ of $G$ with a non-degenerate character $\psi:N\rightarrow \C^\times$. Let $w_G$ 
be a fixed element of $G(F)$ normalizing $T$ and taking $B$ to the opposite Borel $B^-=TN^-$; note that for the action 
of $w_G$ on $T$ as an element in the Weyl group, we have $w_G t w_G^{-1} = w_G(t)$. Similarly, for a standard Levi subgroup $M$ of $G$, let $w_M$
be a fixed element of $M(F)$ normalizing $T$ and taking the positive roots of $T$ in $M$ to 
negative roots of $T$ in $M$.

By a theorem of Rodier, 
 it is known that a principal series representation
 $\pi = \Ind_P^G \sigma$ has a Whittaker model for a character $\psi: N \rightarrow \C^\times$  
if and only if $\sigma$ has a Whittaker model for the character 
$\psi_M: M \cap N \rightarrow \C^\times $ where $\psi_M$ is the character of 
 $M\cap N$ defined by $\psi_M(n) = \psi(w_Gw_Mnw_M^{-1}w_G^{-1})$. It is to be noted that the element $w_G$ (as well as $w_M$) 
is well-defined only up to an element of $T(F)$, and hence the character $\psi_M$ too is well-defined only up to
 conjugation by an element of $T(F)$
which is all that matters for fixing the enhanced Langlands correspondence for $M(F)$.  The 
 enhanced local  Langlands correspondence will be considered  for the groups $G$ and $M$ 
 using a fixed non-degenerate character $\psi: N \rightarrow \C^\times$, and the corresponding non-degenerate 
character $\psi_M: M \cap N \rightarrow \C^\times$ for $M$.

 The enhanced local Langlands correspondence is so setup that the representations $\pi$ and $\sigma$
have under the natural inclusion ${}^LM \hookrightarrow {}^LG$, the same parameters as well as the representations of their naturally isomorphic component groups. 

By Theorem 4.13 of [BT], the natural map from $G(F)$ to $(G/P)(F)$ is surjective, from which one deduces that 
the natural map from $H^1(\Gal(\bar{F}/F), P)$ to  $H^1(\Gal(\bar{F}/F), G)$ is injective. Therefore if $M$ 
is a Levi subgroup in $G$, then the natural maps $H^1(\Gal(\bar{F}/F), M)  \rightarrow 
H^1(\Gal(\bar{F}/F), P) \rightarrow H^1(\Gal(\bar{F}/F), G)$ are all injective. As a consequence, pure innerforms of $M$     give rise to distinct pure innerforms of $G$ in which 
they lie as Levi subgroups. This allows one to associate to a  Vogan $L$-packet on pure innerforms $M_\alpha$ of a quasi-split Levi group $M_0$, 
a Vogan $L$-packet on pure innerforms $G_\alpha$ of a quasi-split group $G_0$
by taking the Langlands quotient of the corresponding principal series representations on $G_\alpha$ with
$M_\alpha$ as the Levi subgroup.

We now come to the following crucial lemma about the contragredient in terms of Langlands quotients. 
\begin{lemma} \label{Lemma 2}
Assume that an irreducible admissible representation $\pi$ of a reductive group $G=G(F)$ is obtained as a Langlands quotient
$\Ind_P^G\sigma \rightarrow \pi$ 
with $\sigma$ 
an irreducible representation of a Levi subgroup $M$ of $P$ which is tempered up to twisting by a 
character of $M$, and that  the exponent of the irreducible representation $\sigma$ of $M$ 
restricted to the maximal split torus in the centre of $M$ is strictly positive. Let $P^-$ be the parabolic `opposite' to $P$, i.e., with 
$P^-\cap P = M$. Let $P'=w_Gw_M(P^-)=w_G(P^-)$. Then $P'$ is a standard parabolic in $G$, 
i.e., contains the Borel subgroup $B$ fixed earlier, and is an associate of $P$ with $M'= w_Gw_M(M)=w_G(M)$.  The contragredient representation $\pi^\vee$ is obtained as a Langlands
quotient 
$\Ind_{P^-}^G\sigma^\vee \rightarrow \pi^\vee$, 
as well as the Langlands quotient of $\Ind_{P'}^G(\sigma'^\vee) \rightarrow \pi^\vee$ where $\sigma'$ is the representation 
of $M' = w_Gw_M(M)$ from the representation $\sigma$ of $M$ through conjugation by $w_Gw_M$.
\end{lemma}

\begin{proof} From [BW], chapter XI, Proposition 2.6 and Corollary 2.7, it follows that the representation  $\pi$ appears in 
$\Ind_P^G\sigma $ 
as a quotient as well as a submodule in $\Ind_{P^-}^G \sigma$; dualizing, we realize  $\pi^\vee$ as a  quotient 
of $\Ind_{P^-}^G\sigma^\vee$.

We note that the principal series representation $\Ind_{P^-}^G\sigma^\vee $ is a standard module with strictly positive exponents, 
so the map $\Ind_{P^-}^G\sigma^\vee \rightarrow \pi^\vee$ realizes $\pi^\vee$ as a Langlands quotient.

We leave the simple checking that $P'=w_Gw_M(P^-)$ is a standard parabolic in $G$ to the reader.  
The natural isomorphism induced by a conjugation action
in $G$ taking $P^-$ to $P'$, $M$ to $M'$ and $\sigma$ to $\sigma'$ proves 
the assertion on $\pi^\vee$ being the Langlands quotient of $\Ind_{P'}^G(\sigma'^\vee) $. 
\end{proof}

\begin{lemma} \label{Lemma 3}Let $\wG$ be a connected reductive algebraic group over $\C$ with a fixed Borel subgroup 
$\wB=\wT \wN$. Let $\wP \supset \wB$ be a standard parabolic with $\wP=\wM \wN$. Fix a 
pinning  $(\wG,{\wB}, \wT, \{X_\alpha\})$   
of $\wG$ which restricts to give a pinning $(\wM,{\wB}_{\wM}, \wT, \{X_\beta\})$   of $\wM$. 
Let $c_{\wG}$ be the automorphism of $\wG$
preserving  the pinning  $(\wG,{\wB}, \wT, \{X_\alpha\})$   
and acting as $t\rightarrow w_{\wG}(t^{-1})$ on $\wT$. 
Similarly, 
let $c_{\wM}$ be the automorphism of $\wM$
preserving  the pinning  $(\wM,\wB_{\wM}, \wT, \{X_\beta\})$   and acting as $t\rightarrow w_{\wM}(t^{-1})$ on $\wT$. 
 Fix representatives of Weyl group elements $w_{\wG}$ and $w_{\wM}$ in $\wG$ and $\wM$ respectively, and 
by an abuse of language, denote them also as $w_{\wG}$ and $w_{\wM}$. Then there is an 
element $t_0 \in \wT$ such that:
$$ t_0 c_{\wM}(m) t_0^{-1}  = w_{\wG} w_{\wM}( c_{\wG}(m))w_{\wM}^{-1}w_{\wG}^{-1} \quad \quad \forall m \in \wM. \quad \quad (\star)$$

\end{lemma}
\begin{proof} Observe that because $w_{\wG}(t) = w_{\wG} w_{\wM} w_{\wM}(t)$ for $t \in \wT$, the two homomorphisms of $\wM$ into $\wG$ 
\begin{eqnarray*} m & \longrightarrow &  c_{\wM}(m), \\
m & \longrightarrow & w_{\wG} w_{\wM}( c_{\wG}(m))w_{\wM}^{-1}w_{\wG}^{-1},
\end{eqnarray*}
preserve $\wT$, and are the same  on $\wT$; further,  
both of them take the root spaces $\alpha$ of $\wT$ in $\wM$ to 
root spaces $-w_{\wG}(\alpha)$. Therefore the two homomorphisms of $\wM$ into $\wG$ differ by conjugation by an element of  $\wT$, proving the lemma. 
\end{proof}

In the next section (Corollary \ref{t_0})
we will improve this lemma to an assertion on
$L$-groups by being more precise about $t_0, w_G,w_M$.

\section{Embedding of $L$-groups}
This section fills up a small detail on embedding of $L$-groups from embedding of their identity component
required to extend Lemma \ref{Lemma 3} to an assertion on
$L$-groups, and then to prove that our Conjecture 2 is compatible with parabolic induction.

We begin by  proving that for a standard (relevant) Levi subgroup $\wM$ inside $\wG$, 
conjugation by $w_{\wG}w_{\wM}$ 
in $\wG$ (for suitably chosen representatives of these elements in the Weyl group) extends to an automorphism of 
 ${}^LG(\C) = \wG \rtimes W_F$ as an $L$-group which is the identity map on $W_F$; only such morphisms of 
$L$-groups  which are identity on $W_F$ will be considered in this section. 
Presumably these considerations have occurred in the literature when
one identifies an associate class of parabolics in a quasi-split group with an associate class of {\it relevant} parabolics 
in the $L$-group. 
 
Let ${}^LG(\C) = \wG \rtimes W_F$ be an $L$-group. 
Denote the action of an element $\sigma \in W_F$ on $g\in \wG$ by $g^\sigma$.
An automorphism of $\wG$ given by the inner conjugation action of $g \in \wG$ need not extend to an automorphism of ${}^LG$. Here is the lemma regarding it.

\begin{lemma} \label{trivial} An automorphism of $\wG$ given by the inner conjugation action of $g_0 \in \wG$ extends to an automorphism of ${}^LG$ (which is identity on $W_F$) if and only if for all $\sigma \in W_F$, $g_0^{\sigma} = g_0 \cdot z(\sigma)$ for some $z(\sigma) \in Z(\wG)$.
\end{lemma}
\begin{proof}Clear! \end{proof}

Next, we see some situations in which we do have elements in $\wG$ satisfying the condition in Lemma \ref{trivial}.

Observe that ${}^LG(\C) = \wG \rtimes W_F$ comes equipped with a $W_F$-invariant pinning $(\wG,{\wB}, \wT, \{X_\alpha\})$   
on $\wG$ which restricts to give a $W_F$-invariant  pinning $(\wM,{\wB}_{\wM}, 
\wT, \{X_\beta\})$   on $\wM$ for all
{\it relevant} Levi subgroups in ${}^LG$.

Since $\wT$ is invariant under $W_F$, its normalizer $N_{\wG}(\wT)$ in $\wG$ is also invariant under $W_F$, giving rise to an action of $W_F$ on the Weyl group 
$W_{\wG}(\wT) 
= N_{\wG}(\wT)/\wT$; more generally, since a relevant Levi subgroup $\wM$ is invariant under 
$W_F$, the normalizer $N_{\wM}(\wT)$ 
of $\wT$ in $\wM$ is also invariant under $W_F$, giving rise to an
action of $W_F$ on  the Weyl group 
$W_{\wM}(\wT) 
= N_{\wM}(\wT)/\wT$.

The Weyl group $W_{\wM}(\wT)$ 
has the element $w_{\wM}$ which takes the Borel subgroup ${\wB}_{\wM}$ to its opposite Borel subgroup. 
We claim that these elements $w_{\wM}$ for any relevant Levi subgroup in ${}^LG$ 
are invariant (as elements of the Weyl group 
$W_{\wM}(\wT)$) under $W_F$. For this it suffices
to note that for any $\sigma \in W_F$, $w_{\wM}^\sigma$ still takes the Borel subgroup ${\wB}_{\wM}$ to its opposite Borel which is clear by
applying $\sigma$ to the equality $w_{\wM}({\wB}_{\wM}) =  {\wB}^-_{\wM}$, and noting that both 
${\wB}_{\wM} = {\wB} \cap \wM $ and  $ {\wB}^-_{\wM} = \wB^- \cap \wM$ are invariant under $W_F$.

\begin{lemma} \label{lemma6} Let ${}^LG(\C) = \wG \rtimes W_F$ be an $L$-group. Any element in the Weyl group of 
$\wG$ which is invariant under $W_F$ is represented by an element in $\wG^{W_F}$ up to $Z(\wG)$.
In particular, the longest element $w_{\wM}$ in a relevant Levi subgroup $\wM$ in ${}^LG$ has a lift to  $N_{\wM}(\wT)$ 
which is $W_F$-invariant up to $Z(\wG)$, i.e. there is a lift $n_{\wM}$ of $w_{\wM}$ to $N_{\wM}(\wT)$ such that for all 
$\sigma \in W_F$,
$n_{\wM}^\sigma = n_{\wM} z(\sigma)$ for some $z(\sigma) \in Z(\wG)$. 
\end{lemma}
\begin{proof}It suffices to prove the lemma assuming that $\wG$ is an adjoint group, and further that it is an
adjoint simple group, allowing us to give a proof case-by-case, since we now have to deal only with groups of type $A_n,D_n,E_6$. It is well-known in these cases that $W^\sigma$ is the Weyl group of $\wG^\sigma$, completing the proof of the lemma.  
\end{proof}

\begin{corollary} \label{t_0}
For a standard (relevant) Levi subgroup $\wM$ inside $\wG$, the element $t_0 \in \wT$ in Lemma \ref{Lemma 3} has the property that $t_0^\sigma/t_0 \in Z(\wM)$ for all $\sigma \in W_F$,  
hence we have a  commutative diagram of homomorphisms of $L$-groups where the vertical arrows are the natural embedding of ${}^LM(\C)$ in ${}^LG(\C)$, the top right arrow is conjugation by $t_0$, and the bottom rightarrow is conjugation by $w_{\wG}w_{\wM}$, both these making sense on $L$-groups by Lemma \ref{trivial}.
$$\xymatrixrowsep{1in}
\xymatrixcolsep{1in}
\xymatrix{  {}^LM(\C) \ar[r]^{c_{\wM}} \ar@{->}[d] 
&  {}^LM(\C)  \ar[r]^{t_0}     & \ar[d] {}^LM(\C)\\
 {}^LG(\C)\ar[r]^{c_{\wG}}  & {}^LG(\C)\ar[r]^{w_{\wG}w_{\wM}} & {}^LG(\C) .
}$$
\end{corollary}
  \begin{proof} Applying $\sigma \in W_F$ to the identity in Lemma \ref{Lemma 3}:
$$ t_0 c_{\wM}(m) t_0^{-1}  = w_{\wG} w_{\wM}( c_{\wG}(m))w_{\wM}^{-1}w_{\wG}^{-1} \quad \quad \forall m \in \wM,$$
and noting that the Weyl group elements $w_{\wG}$ and $w_{\wM}$ have lifts (by Lemma \ref{lemma6}) to $\wG$ which are $\sigma$-invariant
for all $\sigma \in W_F$, it is clear that $t_0^\sigma/t_0 \in Z(\wM)$ for all $\sigma \in W_F$. 

The commutativity of the  diagram now is just this identity from Lemma \ref{Lemma 3} interpreted for $L$-groups.\end{proof}

\begin{proposition}\label{cohomology}
Let $M$ be a quasi-split reductive group over a local field $F$, with $\wM$ the dual group,  $\wT$ the standard
maximal torus in $\wM$, and $\wZ$ the center of $\wM$. Then the natural map on cohomologies:
$H^1(W_F, \wZ) \rightarrow H^1(W_F,\wT)$ is injective.
\end{proposition}
\begin{proof} It suffices to prove the proposition by replacing $W_F$ by its finite quotient $W$  through which 
it operates on $\wM$. By the long exact sequence associated to the short exact sequence of $W$-modules: 
$0\rightarrow \wZ\rightarrow \wT \rightarrow \wT/\wZ\rightarrow 0,$ it suffices to prove that $H^0(W,\wT/\wZ)$ is a connected torus. Note that there is a natural isomorphism
$$\wT/\wZ \cong \prod _{\alpha ~~ {\rm simple}} {\mathbb G}_m,$$
which is $W$-equivariant, where $W$ operates on $\prod _{\alpha ~~{\rm simple}} {\mathbb G}_m,$ by permuting 
co-ordinates. Therefore,
$$(\wT/\wZ)^W  \cong \prod _{\langle \alpha \rangle} {\mathbb G}_m,$$
where the product is taken over the orbits of $W$ on the set of simple roots. This completes the proof of the proposition. \end{proof}

\begin{corollary}\label{identity}
There is a choice of the  element $t_0 \in \wT$ appearing in Lemma \ref{Lemma 3} with,
$$ t_0 c_{\wM}(m) t_0^{-1}  = w_{\wG} w_{\wM}( c_{\wG}(m))w_{\wM}^{-1}w_{\wG}^{-1} \quad \quad \forall m \in \wM,$$
for which $t_0^\sigma = t_0$ for all $\sigma \in W_F$, and hence the top right arrow in the commutative diagram
in Corollary \ref{t_0} can be considered as identity at the level of Langlands parameters. 
\end{corollary}
\begin{proof}By Corollary \ref{t_0} we have $t_0^\sigma/t_0 \in \wZ$ 
for all $\sigma$ in $W_F$, defining a 1-cocycle
of $W_F$ in $\wZ$ which is a co-boundary in $\wT$, hence is a co-boundary in $\wZ$ by Proposition \ref{cohomology},
i.e. there is a $z_0 \in \wZ$ with $t_0^\sigma/t_0 = z_0/z_0^\sigma$ for all $\sigma \in W_F$. Therefore,
$(z_0t_0)^\sigma= (z_0t_0)$   for all $\sigma \in W_F$. One can replace $t_0$ by $z_0t_0$ in the identity:
$$ t_0 c_{\wM}(m) t_0^{-1}  = w_{\wG} w_{\wM}( c_{\wG}(m))w_{\wM}^{-1}w_{\wG}^{-1} \quad \quad \forall m \in \wM,$$
and for such a $t_0$ now with $t_0^\sigma = t_0$ for all $\sigma \in W_F$, the top right arrow in the commutative diagram
above is clearly identity at the level of Langlands parameters (by the definition of equivalence of Langlands parameters). \end{proof}

We finally have the following proposition  we have been after in this section. 

\begin{proposition} \label{prop1}
Assume that an irreducible admissible representation $\pi$ of a reductive group $G=G(F)$ is obtained as a Langlands quotient
$\Ind_P^G\sigma \rightarrow \pi$ 
with $\sigma$ 
an irreducible representation of a Levi subgroup $M$ of $P$ which is tempered up to twisting by a 
character of $M$, and that  the exponents of the representation $\sigma$ of $M$ 
restricted to the maximal split torus in the centre of $M$ are strictly positive.
Then if Conjecture \ref{conj2} holds good for the representation $\sigma$ of $M$, it does also for the representation $\pi$ of $G$.
\end{proposition}
\begin{proof}By Lemma \ref {Lemma 2} and the discussion in the beginning of the section \ref{Rodier}, the enhanced Langlands parameter of $\pi^\vee$ is the same as that of $\sigma^\vee$ conjugated by $w_Gw_M$ in ${}^LG$. By Lemma \ref{lemma6}, the Weyl group element $w_Gw_M$ has a
lift in $\wG$ which is $W_F$-invariant, and therefore, conjugation by $w_Gw_M$ is the identity map on parameters for $G$. 

Since we are assuming that
Conjecture \ref{conj2} is true for tempered representations, the enhanced Langlands parameter of $\sigma^\vee$ with values in ${}^LM$ arises by using the Chevalley involution on ${}^LM$. We are now done regarding the assertion on the Langlands parameter of $\pi^\vee$ 
by the commutative diagram of $L$-groups contained
in Corollary \ref{t_0} combined with Corollary \ref{identity}. Remark 2 fixes the twisting by the character $\eta_{-}$ 
which appears in the conjecture. \end{proof}

\begin{remark}
Most of the difficulties of this section would not be there if our groups were innerforms of split groups instead of quasi-split groups! 
\end{remark}

\begin{remark} 
If $\Ind_P^G\sigma = \pi$, with $\sigma$ a unitary tempered representation of $M(F)$, is irreducible, then again the Langlands parameter of $\sigma$ with values in 
${}^LM$ when considered inside ${}^LG$ through the natural embedding of ${}^LM$ inside ${}^LG$ is the Langlands parameter of $\pi$; further, for $\pi$ to be irreducible, presumably their component groups are isomorphic too. Therefore, once again, Conjecture \ref{conj2} for the representation $\sigma$ of $M(F)$ implies Conjecture \ref{conj2}
for the representation $\pi$ of $G(F)$. 
\end{remark}

\section{Comparing with the work of Adams}

Jeff Adams in \cite {adams} has constructed the contragredient of any representation of $G(\R)$ in terms of 
what he calls the Chevalley involution $C$: an automorphism of $G(\R)$ of order 2 which leaves a fundamental
Cartan subgroup $H_f$ invariant and acts by $h\rightarrow h^{-1}$ on $H_f$. (A fundamental Cartan subgroup 
of $G(\R)$ is a Cartan subgroup of $G$ defined over $\R$  which is of minimal split rank; such Cartan subgroups are conjugate under $G(\R)$.) One of the main theorems of Adams is that $\pi^C \cong \pi^\vee$ for any irreducible representation $\pi$ of $G(\R)$. He uses this to prove that 
every irreducible representation of $G(\R)$ is selfdual if and only if  $-1$ belongs to the Weyl group of $H_f$ in $G(\R)$ (in particular $-1$ belongs to the Weyl group of $G(\C)$).

Assuming $G$ to be a quasi-split reductive group over $\R$, with $(B,T)$ a pair of a Borel subgroup and a 
maximal torus in it, 
 our recipe for constructing the contragredient is in terms of the duality involution $\iota_G$ of $G(\R)$ 
which has the form $\iota_G=\iota_{-} \cdot c_G$ where $\iota_{-}$ in an element in $T^{\ad}(\R)$ which acts by $-1$ 
on all simple root spaces of $T$ in $B$, and $c_G=c_{G,\Pa}$ is the Chevalley 
automorphism of $\Pa=(G,B,T,\{X_\alpha\})$. In this section we 
 check that the duality involution constructed out of considerations with the Whittaker model
is the same as 
that of Adams constructed using fundamental tori.  (This gives credence to our 
recipe on the duality involution $\iota_G$!) Much of the section
uses the `principal $\SL_2$' arising from a regular unipotent element by the Jacobson-Morozov theorem, and transporting 
an analysis inside this principal $\SL_2$  
 to the whole of $G$. Along the way, we also give a proof of 
the theorem of Adams on the existence of Chevalley involution for any reductive group over $\R$.

The following well-known lemma has the effect of saying that the contragredient which appears in the 
character of component group in Conjecture 2 plays no role for real groups.

\begin{lemma} Let $G$ be a connected real reductive group with a Langlands parameter $\varphi: W_{\R} \rightarrow {}^LG.$ 
Then, the group $\pi_0(Z_\varphi)$ is an elementary abelian 2-group.
\end{lemma}
\begin{proof} From the structure of $W_\R = \C^\times \cdot \langle j \rangle$, with $j^2=-1$, 
the group $Z_\varphi$ can be considered as the fixed points of the involution which is $\varphi(j)$,
 on the centralizer of 
$\varphi(\C^\times)$ in $\widehat{G}$. The group $\varphi(\C^\times)$ being connected, 
abelian, and consisting of semisimple elements, its centralizer in $\widehat{G}$ is a connected reductive group. It 
suffices then to prove the following lemma. \end{proof}
\begin{lemma} Let $H$ be a connected reductive group over $\C$, and $j$ an involution on $H$ with $H^j$ its fixed points. 
Then the group of 
connected components  $\pi_0(H^j)$ is an elementary abelian 2-group.
\end{lemma}
\begin{proof}According to \'Elie Cartan, the involution $j$ gives rise to a real structure on $H$, i.e., there is a connected 
real reductive group $H_j$ (with $H$ its complexification) on which $j$ operates, and operates as a Cartan involution on $H_j(\R)$, thus  $H_{j}(\R)^j$ 
is a maximal compact subgroup of $H_j(\R)$, and $H^j$ is the complexification of the compact group $H_{j}(\R)^j$. 
Therefore $\pi_0(H_j(\R))= \pi_0(H_j(\R)^j)= \pi_0(H^j).$ Finally, it suffices to note that the 
group of connected components of an algebraically connected real reductive group is an elementary abelian 2-group, cf. \cite{BT} corollaire 14.5, who attribute the result to Matsumoto. 
\end{proof}

Here is the main result of this section proving the existence of Chevalley involution (in the sense of Adams) for all real reductive groups, and at the same time identifying it to our duality involution $\iota_G$ when $G$ is quasi-split over $\R$.

\begin{proposition}
Let $H_f(\R)$ be a fundamental maximal torus in a real reductive group $G(\R)$. Then there exists an involutive 
automorphism
$C$ of $G(\R)$ which acts as $t\rightarrow t^{-1}$ on $H_f$. Further, if $G(\R)$ is quasi-split with 
a Borel subgroup $B$ of $G(\R)$, and $T$ as a maximal torus in $B$, then we can choose $H_f$ 
as well as a pinning  $(G,B,T,\{X_\alpha\})$ on $G(\R)$ 
so that $C = \iota_G= \iota_{-}\cdot c_G$. 
\end{proposition} 
\begin{proof} First we prove the result assuming $G(\R)$ to be  quasi-split with a Borel subgroup $B$ 
defined over $\R$ containing a maximal torus $T$ defined over $\R$. 

Using the Jacobson-Morozov
theorem, we  construct $\phi: \SL_2(\R) \rightarrow G(\R)$ associated to a regular unipotent element 
inside the Borel subgroup $B$ of $G(\R)$. 
We can assume that $\phi$ commutes with diagram automorphisms of $G(\R)$, i.e. with automorphism group of $\Pa= (G,B,T,\{X_\alpha\})$,
in particular with the diagram automorphism corresponding to $c_G=c_{G,\Pa}$, the Chevalley involution.
We will consider the semi-direct product $G(\R)\rtimes \langle c_G \rangle$. 

Observe that a regular unipotent element in 
$B(\R)$ has nonzero components in each of the (relative) simple root spaces. An element $t_0\in T(\C)$ which operates on all simple 
roots of $B$ by $-1$ is unique up to $Z(\C)$, and therefore can be taken to be the image of 
$j = \left ( \begin{array}{cc} i & 0 \\ 0 & -i\end{array}\right ) \in \SL_2(\C)$ 
up to central elements. 

Let $\Si^1 = \left ( \begin{array}{cc} a & b \\ -b & a\end{array}\right ) \subset \SL_2(\R)$ (with $a^2+b^2=1$). We claim that the centralizer of $\phi(\Si^1)$ inside $G(\C)$ is a fundamental torus, i.e., the centralizer is
a maximal torus  defined over $\R$ such that its split rank over $\R$ is the minimal possible among maximal tori over $\R$.
For this we first prove that the centralizer of $\phi(\Si^1)$ in $G(\C)$ is a maximal torus. 
To prove this, it suffices to prove a similar statement for the 
centralizer in $G(\C)$ of the image of the diagonal torus in $\SL_2(\R)$ which is a standard and a simple result. Next, we prove that the centralizer of $\phi(\Si^1)$ inside $G(\C)$ 
is a fundamental torus. For this, we can assume that $\phi(\Si^1)$ lands inside $K^0$, the connected component of identity of a maximal compact subgroup $K$ of $G(\R)$. Since the 
centralizer of $\phi(\Si^1)$ inside $G(\R)$ is a torus, this is also the case about the centralizer of $\phi(\Si^1)$ inside $K^0$.
Thus the centralizer of $\phi(\Si^1)$ inside $K^0$ is  a maximal torus of $K^0$, but any maximal torus of $G(\R)$ containing a maximal
torus of $K^0$ is a fundamental torus of $G(\R)$ proving that the centralizer of $\phi(\Si^1)$ in $G$ is a fundamental torus. 
 Denote the centralizer of $\phi(\Si^1)$ inside $G(\R)$ as $H_f(\R)$.

The fundamental torus $H_f$  in $G(\R)$ is normalized by $\phi(j)$.  
Since $\phi(\Si^1)$ commutes with $c_G$, $c_G$ leave $H_f$ invariant. Since $c_G$ operates on $T$ as $-w_G$,
and there is an element in $G(\R)$ which operates on $T$ as $w_G$, thus there is an element in 
$G(\R)\rtimes \langle c_G \rangle$ 
which acts by $-1$ on $T$. Since all maximal tori are conjugate in $G(\C)$, there is an element in
$G(\C)\rtimes \langle c_G \rangle$ which acts by $-1$ on $H_f$. Call this element $t_0\cdot c_G$. 
Since 
$c_G$ commutes with $\phi(\Si^1)$, and $t_0\cdot c_G$ acts by $-1$ on $H_f$, it follows that $t_0$ acts by $-1$
on $\phi(\Si^1)$; same for $\phi(j)$. Therefore, $\phi(j)t_0^{-1}$ commutes with $\phi(\Si^1)$, 
so belongs to $H_f(\C)$. Write $\phi(j) = s^{-1}t_0$ with $s \in H_f(\C)$. Then, $t_0\cdot c_G = 
s \phi(j) \cdot c_G$ 
acts by $-1$ on $H_f(\C)$, therefore since $s \in H_f(\C)$, so does    
$\phi(j) \cdot c_G$ which is an automorphism of $G(\R)$ since the inner conjugation action of 
$\phi(j)$ is real, proving the existence of the Chevalley involution (in the sense of Adams, i.e. one which operates on $H_f$ as $t\rightarrow t^{-1}$) 
for quasi-split groups, and at the same time realizing this
Chevalley involution as  $\phi(j) \cdot c_G$ where $c_G$ is the Chevalley involution associated to 
a pinning on $G$.  Since $\phi(j)$ acts by $-1$ on all simple root spaces of $T$ in $B$, $\phi(j)$ can be taken to be 
$\iota_{-}$.

For general $G(\R)$ with $G_0$ a quasi-split innerform, we appeal to a result of Borovoi, cf. [Bor],  according to which  $H^1(\Gal(\C/\R), H_f) \rightarrow 
H^1(\Gal(\C/\R), G_0)$ is surjective. We will use this result of Borovoi for the adjoint group $G_0^{\ad}=G_0/Z$, 
and as a result, any innerform of $G_0(\R)$ is obtained by twisting by a cocycle
with values in $(H_f/Z)(\C)$. It can be easily seen that if $G_c$ is obtained by twisting $G_0(\R)$ by an element
$c \in H_f(\C)$ with $ c\bar{c} \in Z(\C)$, then if $C$ is a Chevalley involution (in the sense of Adams) on $G_0(\R)$, 
$C': x \rightarrow C(c^{-1}xc)$ is 
a Chevalley involution on $G_c(\R)$ acting as $-1$ on $H_f$ which continues to be a fundamental torus in $G_c(\R)$. 
We check that $C'$ is defined over $\R$ for the real structure defining $G_c$. If $x\rightarrow \bar{x}$ is the original
complex conjugation, and $x\rightarrow \underline{x}$ is associated to the real structure $G_c$, then by definition of twisting, $\underline{x}=c\bar{x}c^{-1}$. To check that $C'$ is defined over $\R$, we need to check that $C'(\underline{x})= \underline{C'(x)}$, i.e., 
$$C'(c \bar{x}c^{-1})  = c \overline{C'(x)} c^{-1}.$$ 
Using the definition of $C'(x) = C(c^{-1}xc) = cC(x) c^{-1}$, we can rewrite the above as, 
 $$C( \bar{x}) = c \bar{c} \overline{C(x)} \bar{c}^{-1} c^{-1},$$ 
but since $c\bar{c} \in Z(C)$, and $C$ is defined over $\R$, this equation holds.
\end{proof}

\vspace{.2cm}

\noindent{\bf Acknowledgement:} This work is of older vintage, and was part of my work [DP], which is still incomplete. 
The work of Kaletha [Ka] already contains Conjecture 2 who arrived at it independently. Since this work and that of Kaletha
 do not have much else in common, it is hoped that this paper may still have some utility. There is also the work of Adams and Vogan [AV] 
formulating and proving (for $F = \R$) Conjecture \ref{conj2} just for Langlands parameters. A recent preprint [Li] 
of W-W. Li deals with the recipe on contragredient in the positive characteristic. The author thanks one of the referees of this paper for a very meticulous job.



\begin{thebibliography}{99}

\bibitem[Ad]{adams} 
J. Adams: The real Chevalley involution.  {\it Compos. Math. 150} (2014), no. 12, 2127--2142. 



\bibitem[AV]{adams-vo}J. Adams and D. A. Vogan Jr: 
Contragredient representations and characterizing the local Langlands correspondence. {\it Amer. J. Math.}
 138 (2016), no. 3, 657--682. 



\bibitem[BT]{BT}A. Borel and J. Tits: Groupes Reductifs, {\it Publ. Math. IHES} 27 (1965) 55-150.

\bibitem[BW]{BW} A. Borel and N. Wallach: Continuous cohomology, discrete subgroups, and representations of reductive groups. Second edition. Mathematical Surveys and Monographs, 67. American Mathematical Society, Providence, RI, 2000. 

\bibitem[Bor]{Bor}M. Borovoi: Galois cohomology of reductive algebraic groups over the field of real numbers; arXiv:1401.5913. 

\bibitem[GGP]{GGP} W.T. Gan, B.H. Gross, Benedict H.; D. Prasad, 
Symplectic local root numbers, central critical L values, and restriction problems in the representation theory of classical groups. 
{\it Astérisque No.}
 346 (2012), 1–-109


\bibitem[HO]{HO} V. Heiermann, E. Opdam, 
On the tempered L-functions conjecture.
{\it Amer. J. Math.}  135 (2013), no. 3, 777--799.

\bibitem[Ka]{kal} T. Kaletha: Genericity and contragredience in the local Langlands correspondence. {\it Algebra and Number Theory}
 7 (2013), no. 10, 2447--2474.

\bibitem[Li]{Li} W-W. Li: Contragredient representations over local fields of positive characteristic, 
arXiv:1802.08999 [math.RT].

\bibitem[LST]{LST} Y. Lin, B. Sun, and S. Tan: MVW-extensions of quaternionic classical groups. {\it Math. Z.} 
277 (2014), no. 1--2, 81--89. 
\bibitem[MVW]{MVW} C. Moeglin, M-F. Vign\'eras, and J.-L. Waldspurger:
Correspondances de Howe sur un corps p-adique. 
{\it Lecture Notes in Mathematics, 1291.} Springer-Verlag, Berlin, 1987.

\bibitem[DP]{DP}D. Prasad: A `relative' local Langlands correspondence,   arXiv:1512.04347.


\bibitem[Vo]{Vo} D. A. Vogan:
 The local Langlands conjecture. 
Representation theory of groups and algebras, 305--379, Contemp. Math., 145, Amer. Math. Soc., Providence, RI, 1993. 

\bibitem[Vo2]{Vo2} D. A. Vogan: Gelfand-Kirillov dimension for Harish-Chandra modules, {\it 
Invent. Math.,} Vol. 48, 1978, pp. 75--98.

\end{thebibliography}
\end{document}